\documentclass[11pt,reqno,tbtags,a4paper]{amsart}
\usepackage{amssymb}
\usepackage{xpunctuate}
\usepackage{url}
\usepackage[square,numbers]{natbib}
\bibpunct[, ]{[}{]}{;}{n}{,}{,}

\title
{A tightness criterion for fragmentations}

\date{7 July, 2025;
revised 18 June, 2026}

\author{Gabriel Berzunza Ojeda}
\address{Department of Mathematical Sciences, University of Liverpool, United
  Kingdom} 
\email{gabriel.berzunza-ojeda@liverpool.ac.uk} 
\urladdr{https://www.liverpool.ac.uk/mathematical-sciences/staff/gabriel-berzunza-ojeda/}

\author{Cecilia Holmgren} 
\address{Department of Mathematics, Uppsala University, Sweden} 
\email{cecilia.holmgren@math.uu.se} 
\urladdr{https://katalog.uu.se/empinfo/?id=N5-824}

\author{Svante Janson}
\thanks{Supported by the Knut and Alice Wallenberg Foundation; 
 Ragnar Söderberg's Foundation; 
the Swedish Research Council.  
}
\address{Department of Mathematics, Uppsala University, 
Sweden}
\email{svante.janson@math.uu.se}
\urladdr{https://www2.math.uu.se/$\sim$svantejs}


\numberwithin{equation}{section}

\renewcommand\le{\leqslant}
\renewcommand\ge{\geqslant}
\renewcommand\leq{\leqslant}
\renewcommand\geq{\geqslant}

\allowdisplaybreaks

\theoremstyle{plain}
\newtheorem{theorem}{Theorem}[section]
\newtheorem{lemma}[theorem]{Lemma}
\newtheorem{proposition}[theorem]{Proposition}
\newtheorem{corollary}[theorem]{Corollary}

\theoremstyle{definition}

\newcommand\xqed[1]{%
    \leavevmode\unskip\penalty9999 \hbox{}\nobreak\hfill
    \quad\hbox{#1}}

\newtheorem{exampleqqq}[theorem]{Example}
\newenvironment{example}{\begin{exampleqqq}}
  {\xqed{$\triangle$}\end{exampleqqq}}

\newtheorem{remarkqqq}[theorem]{Remark}
\newenvironment{remark}{\begin{remarkqqq}}
  {\xqed{$\triangle$}\end{remarkqqq}}

\newtheorem*{acks}{Acknowledgements}

\theoremstyle{remark}

\newcounter{dummy}
\makeatletter
\newcommand\myitem[1][]{\item[#1]\refstepcounter{dummy}\def\@currentlabel{#1}}
\makeatother

\newenvironment{romenumerate}[1][-10pt]{
\addtolength{\leftmargini}{#1}\begin{enumerate}
 \renewcommand{\labelenumi}{\textup{(\roman{enumi})}}%
 }{\end{enumerate}}

\newcounter{oldenumi}
{\setcounter{oldenumi}{\value{enumi}}
\begin{romenumerate} \setcounter{enumi}{\value{oldenumi}}}
{\end{romenumerate}}

\newcounter{thmenumerate}
\newenvironment{thmenumerate}
{\setcounter{thmenumerate}{0}%
 \def\item{\par
 \refstepcounter{thmenumerate}\textup{(\roman{thmenumerate})\enspace}}
}
{}

\newcounter{xenumerate}   

\newcommand\pfitemx[1]{\par#1:}
\newcommand\pfitemref[1]{\pfitemx{\ref{#1}}}

\newcommand{\refT}[1]{Theorem~\ref{#1}}
\newcommand{\refTs}[1]{Theorems~\ref{#1}}
\newcommand{\refC}[1]{Corollary~\ref{#1}}
\newcommand{\refCs}[1]{Corollaries~\ref{#1}}
\newcommand{\refL}[1]{Lemma~\ref{#1}}

\newcommand{\refR}[1]{Remark~\ref{#1}}

\newcommand{\refS}[1]{Section~\ref{#1}}
\newcommand{\refSs}[1]{Sections~\ref{#1}}

\newcommand{\refApp}[1]{Appendix~\ref{#1}}

\begingroup
  \divide\count255 by 60
  \multiply\count255 by -60
  \advance\count255 by \time
  \ifnum \count255 < 10 \xdef\klockan{\the\count1.0\the\count255}
  \else\xdef\klockan{\the\count1.\the\count255}\fi
\endgroup

\newcommand{\sumi}{\sum_{i=1}^\infty}

\newcommand{\sumik}{\sum_{i=1}^k}

\newcommand\set[1]{\ensuremath{\{#1\}}}
\newcommand\bigset[1]{\ensuremath{\bigl\{#1\bigr\}}}
\newcommand\Bigset[1]{\ensuremath{\Bigl\{#1\Bigr\}}}

\newcommand\xpar[1]{(#1)}
\newcommand\bigpar[1]{\bigl(#1\bigr)}
\newcommand\Bigpar[1]{\Bigl(#1\Bigr)}

\newcommand\bigsqpar[1]{\bigl[#1\bigr]}

\newcommand\cpar[1]{\{#1\}}

\newcommand\bigabs[1]{\bigl\lvert#1\bigr\rvert}

\def\rompar(#1){\textup(#1\textup)}    

\def\xexp(#1){e^{#1}}

\newcommand\floor[1]{\lfloor#1\rfloor}

\newcommand\ntoo{\ensuremath{{n\to\infty}}}

\newcommand\downto{\searrow}

\newcommand{\tend}{\longrightarrow}
\newcommand\dto{\overset{\mathrm{d}}{\tend}}
\newcommand\pto{\overset{\mathrm{p}}{\tend}}

\newcommand\fdto{\overset{\mathrm{f.d.}}{\tend}}

\newcommand\eqd{\overset{\mathrm{d}}{=}}

\newcommand\Op{O_{\mathrm p}}

\newcommand\bbR{\mathbb R}

\newcommand\bbN{\mathbb N}

\newcounter{CC}
\newcounter{cc}

\newcommand\E{\operatorname{\mathbb E}{}} 
\renewcommand\P{\operatorname{\mathbb P{}}}

\newcommand\Exp{\operatorname{Exp}}

\newcommand\ga{\alpha}

\newcommand\gs{\sigma}

\newcommand\gU{\Upsilon}
\newcommand\eps{\varepsilon}
\renewcommand\phi{\xxx}  

\newcommand\cF{\mathcal F}

\newcommand\cT{{\mathcal T}}

\newcommand\cV{\mathcal V}

\newcommand\indic[1]{\boldsymbol1\cpar{#1}}

\newcommand\etta{\boldsymbol1}

\newcommand\qw{^{-1}}

\newcommand\oi{\ensuremath{[0,1]}}

\newcommand\ooo{[0,\infty)}

\newcommand\lhs{left-hand side}
\newcommand\rhs{right-hand side}

\newcommand\xoo{_1^\infty}

\newcommand\ixoo{_{i=1}^\infty}

\newcommand\nn{^{(n)}}

\newcommand\SSSx{\mathbb S}
\newcommand\SSS{\mathbb S^\downarrow}
\newcommand\SSSoi{\SSS_{\le1}}
\newcommand\SSSi{\SSS_1}
\newcommand\SSSoo{\SSS_{<\infty}}
\newcommand\YY{Y}
\newcommand\bft{\mathbf{t}}
\newcommand\hT{\widehat{T}}
\newcommand\hX{\widehat{X}}
\newcommand\TPL{\Upsilon}
\newcommand\TPLn{\TPL\nn}
\newcommand\bp{\mathbf{p}}
\newcommand\bw{\mathbf{w}}

\newcommand\cadlag{c\`adl\`ag}

\begin{document}

\begin{abstract} 
This note presents a simple criterion for the tightness of stochastic fragmentation processes.  Our work is motivated by an application to a fragmentation process derived from deleting edges in a conditioned Galton-Watson tree \cite{BOH2023}. In that paper, while finite-dimensional convergence was established, the claimed functional convergence relied on \cite[Lemma 22]{BroutinMarckert}, which is unfortunately incorrect. We show how our results can correct the proof in \cite{BOH2023}. 

Furthermore, we show the applicability of our results by establishing tightness for fragmentation processes derived from various random tree models previously studied in the literature, such as Cayley trees, trees with specified degree sequences, and $\mathbf{p}$-trees.
\end{abstract}

\maketitle

\section{Introduction}

We consider fragmentation processes. 
Intuitively, a (deterministic or random) fragmentation process describes how
a mass is split into smaller and smaller pieces as time increases.
We are only interested in the masses (or sizes) of the pieces, which are
positive real numbers that we arrange in decreasing order.
(We use ``positive'', ``increasing'' and ``decreasing'' in the weak sense.)
The number of pieces may be finite or (countably) infinite, but we will
always  write the masses as an infinite sequence $(x_i)\xoo$, extending a finite
sequence by zeroes.
The space of possible mass configurations
at any given time $t\ge0$ is thus
\begin{align}\label{SSS}
  \SSS:=\bigset{(x_i)\xoo:x_1\ge x_2\ge \dots \ge0}.
\end{align}
(We are usually interested only in smaller subspaces of this space, see
\refS{Snot}, but we try to state results in great generality.)
We define also a partial order $\preceq$ on $\SSS$; the intuitive meaning
of $(y_i)\xoo\preceq(x_i)\xoo$ is that the fragmentation
$(y_i)\xoo$ can be obtained from $(x_i)\xoo$ by further fragmenting each
of the pieces in some way; see \eqref{preceq} for a formal definition.

A (deterministic or stochastic) \emph{fragmentation process} is a 
function $X(\cdot)=(X_i(\cdot))\xoo$ from $\bbR_+:=\ooo$ to $\SSS$, which is \cadlag{}
(for the product topology on $\SSS$, see \refS{Snot}), and furthermore is
decreasing for the order $\preceq$; in other words,
\begin{align}
  \label{preceq0}
X(t_2) \preceq X(t_1)
\qquad\text{when $0\le t_1\le t_2$}.
\end{align}
Note that we do not assume any other properties, such as, for example, Markov.
(We state all results for fragmentation processes
defined on $\ooo$; the results hold also for processes defined on $\oi$,
\emph{mutatis mutandis}.)

The purpose of this note is to state and prove a simple criterion for the tightness of a sequence of stochastic fragmentation processes, see \refS{Smain}. This result is primarily intended for applications where finite-dimensional convergence of the processes can already be established;
if one also can show that our condition holds,
then one can conclude that the processes converge
in the Skorohod space $D(\bbR_+,\SSS)$
(\refT{T2}).

The motivation for this note is such an application to a certain
fragmentation process derived from deleting edges in a random tree
\cite[Lemma 5]{BOH2023}.
In that paper first finite-dimensional convergence was proved,
and then convergence in $D(\bbR_+,\SSS)$ was claimed 
using a general compactness result in  \cite[Lemma 22]{BroutinMarckert} --
however, this lemma in \cite{BroutinMarckert} 
is unfortunately incorrect (see \refS{Swrong}). In \refS{Swrong}, we provide a correct proof of \cite[(2) of Theorem 3]{BroutinMarckert}. Furthermore, in \refS{Scorrect}, we show how the proof in \cite{BOH2023} can be corrected using the results established below.

We state our general results for fragmentation processes in \refS{Smain},
with proofs in \refS{Spf}. In the remaining sections, we apply these
findings to the fragmentation of trees. Specifically, we present general
results in \refS{Stree} and their weighted versions in \refS{Sweight}. In
\refSs{Scorrect}--\ref{Sfixed} and \ref{ptreesApp}, we demonstrate the
applicability of our results by establishing tightness for fragmentation
processes derived from various random tree models previously studied in the
literature (such as Cayley trees, trees with specified degree sequences, and 
$\mathbf{p}$-trees), 
for which only convergence of finite-dimensional distributions had been established and tightness had been left unproven.

\section{Notation}\label{Snot}

Recall the definition \eqref{SSS} of $\SSS$.
We equip $\SSS$ with the product topology, i.e., the topology as a subspace of
$\bbR^\infty$. 
Convergence $(x\nn_i)\xoo\to(x_i)\xoo$ is thus equivalent to $x\nn_i\to x_i$
for each $i$.
Then $\SSS$ is a Polish space, i.e., a complete separable metric space.
There are different metrics that generate the product topology;
one convenient standard choice is
\begin{align}\label{vd1}
  d\bigpar{(x_i)\xoo,(y_i)\xoo}:=\sumi 2^{-i}\min\bigpar{|x_i-y_i|,1}.
\end{align}
(Note that $\SSS$ is a closed subset of $\bbR^\infty$, and thus complete for
this metric.)

We consider also the subspaces
\begin{align}\label{SSSoi}
  \SSSoi&:=\Bigset{(x_i)\xoo\in\SSS: 0\le\sumi x_i \le1},
\\\label{SSSi}
  \SSSi&:=\Bigset{(x_i)\xoo\in\SSS: \sumi x_i =1}.
\end{align}
These are given the subspace topologies, defined by the same metric \eqref{vd1}.
Some useful equivalent metrics on $\SSSoi$ and $\SSSi$ are discussed in
\refL{L0}.

Let $\bbN:=\set{1,2,3,\dots}$. 

As said in the introduction, we define  a partial order on $\SSS$; the
formal definition is
\begin{align}
  \label{preceq}
(y_i)\xoo \preceq(x_i)\xoo
\iff
\exists\ga:\bbN\to\bbN 
\text{ such that }
\sum_{j:\ga(j)=i}y_j \le x_i 
\text{ for all }
i\in\bbN.
\end{align}
The intuitive interpretation is that the piece corresponding to $y_j$ is 
a fragment of the piece corresponding to $x_{\ga(j)}$;
in other words, the piece with mass $x_i$ is fragmented into pieces with
masses $(y_j:j\in\ga\qw(i))$.
Note the inequality in the definition \eqref{preceq}; we allow mass to
disappear (or, equivalently, becoming invisible dust); 
in particular, $\ga\qw(i)$ may be empty.
(For processes in $\SSSi$, strict inequality is not possible so
there is no dust.)

Recall that we define
a (deterministic or stochastic) \emph{fragmentation process} as a 
function $X(\cdot)=(X_i(\cdot))\xoo$ from $\bbR_+:=\ooo$ to $\SSS$, which is \cadlag{} 
in the topology above, and decreasing for the order $\preceq$, i.e.,
\eqref{preceq0} holds.

For any metric space $\SSSx$, the space $D(\bbR_+,\SSSx)$
consist of all \cadlag{} functions $\bbR_+\to\SSSx$, equipped with the Skorohod
$J_1$ topology;
if $\SSSx$ is complete and separable, then so is
$D(\bbR_+,\SSSx)$, see e.g.\
\cite[Chapter 3]{Billingsley2ed} 
or
\cite[Chapter 3]{EthierKurtz}.
Note that equivalent metrics in $\SSSx$ define equivalent metrics and thus
the same topology in  
$D(\bbR_+,\SSSx)$, e.g.\ as a consequence of
\cite[Problem 3.11.13]{EthierKurtz}.

Given a stochastic process $X(\cdot)$, we say 
that a random variable $\tau$ is an \emph{$X$-stopping time} if 
it is a stopping time with respect to the natural filtration 
generated by $X$, i.e., the filtration
$(\cF_t)_{t\ge0}$ with $\cF_t:=\gs\set{X(s): s\le t}$, $t\ge0$.

For stochastic processes $X\nn(\cdot)$ and $X(\cdot)$,
\emph{finite-dimensional convergence}, denoted $X_n(\cdot)\fdto X(\cdot)$,
means joint convergence in distribution $(X_n(t_i))_{i=1}^m\dto (X(t_i))_{i=1}^m$
for any fixed finite sequence $t_1,\dots,t_m$.
(In general, it suffices that this holds for $t_1,\dots t_m$ in the dense
set of times $t$ where $X(t)$ a.s.\ does not jump. In our applications
this is the entire $\bbR_+$.)

If $X_n$ are random variables and $a_n$ positive numbers, then
$X_n=\Op(a_n)$ means
that the variables $X_n/a_n$ are bounded in probability (tight).

Unspecified limits are as \ntoo.

\section{Main results}\label{Smain}

One of our main results is the following, which is proved in \refS{Spf},
based on Aldous's  criterion for tightness. 

\begin{theorem}\label{T1}
  Let $X\nn(\cdot)=(X\nn_i(\cdot))\ixoo$ be a sequence of
fragmentation processes with $X\nn(t)\in\SSS$,
and define
\begin{align}\label{qn}
  Q\nn(t):=\sumi \bigpar{X_i\nn(t)}^2.
\end{align}
Assume that $Q\nn(0)<\infty$ (which implies $Q\nn(t)<\infty$ for every $t\ge0$).
Suppose further that 
\begin{romenumerate}
\item \label{T1i}
for every fixed $t\ge0$, the family $X\nn(t)$ of random variables in $\SSS$
is tight, and
\item \label{T1ii}
for every uniformly bounded sequence of
$X^n$-stopping times $\tau_n$ and every sequence $h_n$ of positive numbers
with $h_n\to0$, we have, as \ntoo,
\begin{align}
  \label{qn2p}
Q\nn(\tau_n) -Q\nn(\tau_n+h_n) \pto0.
\end{align}
\end{romenumerate}
Then,  the family of processes
$X\nn(\cdot)$ is tight in $D(\bbR_+,\SSS)$.
\end{theorem}
It is easy to see, and verified in detail in \refL{L2} below,
that $Q\nn(t)$ is a decreasing function of $t$, and thus 
$Q\nn(\tau_n)\ge Q\nn(\tau_n+h_n)$.
\refT{T1} thus has the following corollary.

\begin{corollary}\label{CT1}
  With the notation in \refT{T1}, 
suppose that 
\ref{T1i} holds, and furthermore that
for every uniformly bounded sequence of
$X^n$-stopping times $\tau_n$ and every sequence $h_n$ of positive numbers
with $h_n\to0$, we have, as \ntoo,
\begin{align}
  \label{qn2}
\E Q\nn(\tau_n) -\E Q\nn(\tau_n+h_n) \to0,
\end{align}
where we assume  that the \lhs{} 
is not of the form $\infty-\infty$. 
Then,  the family of processes
$X\nn(\cdot)$ is tight in $D(\bbR_+,\SSS)$.
\end{corollary}  

As a simple consequence, we obtain the following, also proved in \refS{Spf}.

\begin{theorem}\label{T2}
  Let $X\nn(\cdot)=(X\nn_i(\cdot))\ixoo$ be a sequence of
fragmentation processes with $X\nn(t)\in\SSS$, and suppose that we have
finite-dimensional convergence
\begin{align}
  X\nn(\cdot)\fdto X(\cdot)
\end{align}
for some process $X(\cdot)\in D(\bbR_+,\SSS)$.
Define $Q\nn(t)$ by \eqref{qn} and suppose that the condition 
\eqref{qn2p} (or the stronger \eqref{qn2}) holds
for every uniformly bounded sequence of
$X^n$-stopping times $\tau_n$ and every sequence $h_n$ of positive numbers
with $h_n\to0$.
Then,  
$X\nn(\cdot)\dto X(\cdot)$  in $D(\bbR_+,\SSS)$.
\end{theorem}

We have defined fragmentation processes as processes taking values in
$\SSS$, and therefore 
the results above are stated for $D(\bbR_+,\SSS)$. 
In applications we frequently have 
fragmentation processes taking values in $\SSSoi$ or $\SSSi$, and then the
results above often may be combined with the
following simple lemmas
to conclude the corresponding results in
$D(\bbR_+,\SSSoi)$ or $D(\bbR_+,\SSSi)$. 
Note that it is potentially more useful to use these spaces since, as shown
in \refL{L0} below, we may use stronger metrics on  $\SSSoi$ and $\SSSi$ than on
$\SSS$, for example the $\ell^2$ and $\ell^1$ metrics, respectively.

\begin{remark}
 If the fragmentation processes $X\nn(\cdot)$ take their values in $\SSSoi$
 or $\SSSi$, then condition \ref{T1i} in \refT{T1}
holds automatically, since $\SSSoi$ is compact,
and thus \eqref{qn2p} or \eqref{qn2} is sufficient for tightness.
\end{remark}

\begin{lemma}\label{LSSSoi}
Let $X\nn(\cdot)=(X\nn_i(\cdot))\ixoo$ be a sequence of
fragmentation processes with $X\nn(t)\in\SSSoi$.
\begin{romenumerate}
  \item \label{LSSSoi1}
  The family of processes
$X\nn(\cdot)$ is tight in $D(\bbR_+,\SSSoi)$
if and only if it is tight in $D(\bbR_+,\SSS)$.
\item \label{LSSSoi2}
Suppose further that $X(\cdot)$ is a process in $D(\bbR_+,\SSSoi)$.
Then
$X\nn(\cdot)\dto X(\cdot)$  in $D(\bbR_+,\SSSoi)$
if and only if
$X\nn(\cdot)\dto X(\cdot)$  in $D(\bbR_+,\SSS)$.
\end{romenumerate}
\end{lemma}
\begin{proof}
  Immediate, since $\SSSoi$ is a closed subspace of $\SSS$, 
and thus $D(\bbR_+,\SSSoi)$ is a closed subspace of $D(\bbR_+,\SSS)$
with the subspace topology.
(This is a consequence of e.g.\ \cite[Proposition 3.5.3]{EthierKurtz}.)
\end{proof}

\begin{lemma}\label{LSSSi}
Let $X\nn(\cdot)=(X\nn_i(\cdot))\ixoo$ be a sequence of
fragmentation processes with $X\nn(t)\in\SSSi$.
Suppose further that $X(\cdot)$ is a process in $D(\bbR_+,\SSSi)$.
Then
$X\nn(\cdot)\dto X(\cdot)$  in $D(\bbR_+,\SSSi)$
if and only if
$X\nn(\cdot)\dto X(\cdot)$  in $D(\bbR_+,\SSS)$,
and also if and only if $X\nn(\cdot)\dto X(\cdot)$  in $D(\bbR_+,\SSSoi)$.
\end{lemma}
\begin{proof}
  Immediate, since $\SSSi$ is a subspace of $\SSSoi$ and $\SSS$, 
and thus $D(\bbR_+,\SSSi)$ is a  subspace of both $D(\bbR_+,\SSSoi)$
and $D(\bbR_+,\SSS)$, with the subspace topology.
(Again by e.g.\ \cite[Proposition 3.5.3]{EthierKurtz}.)
\end{proof}

\begin{lemma}\label{L0}
  \begin{thmenumerate}
  \item \label{L0SSSoi}
On $\SSSoi$,  
the $\ell^q$ metric is equivalent to 
the product metric \eqref{vd1}
for any $q\in(1,\infty]$.
$\SSSoi$ is a compact metric space (for any of these equivalent metrics).

  \item \label{L0SSSi}
On $\SSSi$,  
the $\ell^1$ metric is equivalent to 
the product metric \eqref{vd1}.
(And hence so is the $\ell^q$ metric for any $q\in[1,\infty]$.)
The metric space $(\SSSi, \ell^1)$  is complete and separable, and
thus $\SSSi$ 
is a  Polish space, which is non-compact.
  \end{thmenumerate}
\end{lemma}
\begin{proof}[Sketch of proof]
\pfitemref{L0SSSoi}  
Compactness follows since $\SSSoi$ is a closed subset of $\oi^\infty$.
Convergence in the metric \eqref{vd1}, and thus pointwise, implies
convergence in $\ell^q$ for any $q\in(1,\infty)$ by dominated convergence,
using the fact that $(x_i)_{i\geq 1}\in\SSSoi$ implies $0\le x_i\le 1/i$ for every
$i$.

\pfitemref{L0SSSi}
Elements of $\SSSi$ can be regarded as probability measures on $\bbN$.
Hence convergence in $\SSSi$ in the metric \eqref{vd1}, which as noted above
means pointwise convergence,
implies convergence in $\ell^1$ by 
Scheff{\'e}'s lemma \cite[Theorem 5.6.4]{Gut}.
The converse is obvious, and thus 
the metric $d$ is equivalent to the $\ell^1$ metric on $\SSSi$.
Completeness follows since $\SSSi$ is a closed subset of
the Banach space $\ell^1$, and separability follows since $\SSS$ (or
$\ell^1$) is separable. 

 If we define 
 \begin{align}\label{bad}
x\nn_i:=\frac{1}{n}\indic{i\le n}
\qquad\text{and}\qquad
x\nn:=(x\nn_i)\xoo\in\SSSi,   
 \end{align}
then $x\nn\to0$ in $\SSS$ (and thus in $\SSSoi$) 
in the product metric \eqref{vd1}, but not in $\ell^1$. 
This example shows that
$\SSSi$ is not a closed subset of $\SSSoi$ or $\SSS$, and hence
$\SSSi$ is not compact.
It shows also that
the $\ell^1$ metric is not equivalent to the product metric \eqref{vd1} on $\SSSoi$.
\end{proof}

\begin{remark}
Note that there is no result for $\SSSi$ corresponding to
\refL{LSSSoi}\ref{LSSSoi1}:
even if all $X\nn(t)\in\SSSi$, 
we cannot conclude tightness in $D(\bbR_+,\SSSi)$ without further assumptions.
(This is because $\SSSi$ is not a closed subspace of $\SSSoi$ and $\SSS$.)
A counterexample is provided by $X\nn(t):=x\nn$ in \eqref{bad} for all $t\ge0$,
or 
(if we want $X\nn(0)$ fixed)
by $X\nn(t):=x\nn$ for $t \ge U$ and $X\nn(t):=(1,0,0,\dots)$ for $0\le
t<U$, with $U\sim \Exp(1)$, say.
\end{remark}

\begin{remark}\label{RSSSoo}
  Yet another alternative, used in for example  \cite{BOH2023} 
is to consider fragmentation processes as taking values in the space 
\begin{align}\label{SSSoo}
\SSSoo := \Bigset{(x_i)\xoo\in\SSS: 0\le\sumi x_i < \infty},
\end{align}
equipped with the $\ell^1$ metric.
Note that $\SSSi$ is a closed subspace of $\SSSoo$.
(See \refL{L0}.)
Consequently, analogously to \refL{LSSSi},
for processes taking values in $\SSSi$,
tightness or convergence in $D(\bbR_+,\SSSi)$ 
is equivalent to
tightness or convergence in $D(\bbR_+,\SSSoo)$. 
\end{remark}

\section{Proof of  main results}\label{Spf}

We begin with two simple lemmas.

\begin{lemma}\label{L1}
Suppose that $(y_i)\xoo\in\SSS$ and that $x\ge\sum_{i \geq 1} y_i$. Then
\begin{align}
  \label{l1}
x(x-y_1) \le x^2 -\sumi y_i^2.
\end{align}
\end{lemma}
\begin{proof}
  We have
  \begin{align}
    x^2-x(x-y_1) 
= xy_1 \ge \sumi y_iy_1 \ge \sumi y_i^2.
  \end{align}
\end{proof}

The next lemma applies to both deterministic and random fragmentation
processes.

\begin{lemma}\label{L2}
Suppose that $X(\cdot)=(X_i(\cdot)) \ixoo$  is a fragmentation process, and 
let, for $k\ge1$,
\begin{align}\label{qsk}
  S_k(t):=\sumik X_i(t)
\end{align}
be the sum of the $k$ largest masses.
Define
\begin{align}\label{qt}
Q(t):=\sumi [X_i(t)]^2.
\end{align}
Then, for any two times $0\le t_1\le t_2$ 
we have
\begin{align}\label{l2q}
  Q(t_1)\ge Q(t_2)
\end{align}
and, for every $k\ge1$, 
\begin{align}\label{l2}
0\le  S_k(t_1)-S_k(t_2) 
\le
2\sqrt{k\xpar{Q(t_1)-Q(t_2)}}.
\end{align}
 \end{lemma}
 \begin{proof}
By our definition of a fragmentation process,
we have $(X_i(t_2))\xoo \preceq (X_i(t_1))\xoo$.
Let $\ga:\bbN\to\bbN$ be a corresponding function as in the definition
\eqref{preceq} and let $J_i:=\ga\qw(i)$, $i\ge1$.
Then $(J_i)_{i \geq 1}$ is a partition of $\bbN$, and
\begin{align}\label{vd2}
\sum_{j\in J_i} X_j(t_2)\le X_i(t_1).  
\end{align}

For the first inequality in \eqref{l2}, let
$I_k:=\set{\ga(j):j\le k}$.
Then $\set{1,\dots,k}\subseteq\bigcup_{i\in I_k}J_i$, and thus by summing
\eqref{vd2} for $i\in I_k$,
\begin{align}\label{vd22}
  S_k(t_2) 
\le \sum_{i\in I_k} \sum_{j\in J_i} X_j(t_2)
\le \sum_{i\in I_k}  X_i(t_1)
.\end{align}
The set $I_k$ has at most $k$ elements, and thus the 
\rhs{} of \eqref{vd22} is the sum of at most $k$ masses at time $t_1$,
whence it is at most $S_k(t_1)$. 
Hence $S_k(t_2)\le S_k(t_1)$ as we wanted to show.

We turn to the second inequality in \eqref{l2}.
Let, for each $i\ge1$,
\begin{align}\label{vd3}
    \YY_i := \max_{j\in J_i} X_j(t_2),
\end{align}
interpreted as $0$ if $J_i=\emptyset$.
Note that the sets $J_i$ are disjoint; thus $\YY_1,\dots,\YY_k$ are the
masses of $k$ different fragments at time $t_2$, and thus
$\sumik \YY_i \le S_k(t_2)$.
Consequently,
\begin{align}\label{vd4}
  S_k(t_1)-S_k(t_2) \le S_k(t_1) -\sumik \YY_i = \sumik (X_i(t_1)-\YY_i).
\end{align}
Furthermore, 
by \refL{L1} applied to $(X_j(t_2):j\in J_i)$, 
using \eqref{vd2} and \eqref{vd3},
\begin{align}\label{vd5}
  X_i(t_1)\bigpar{X_i(t_1)-\YY_i}
\le X_i(t_1)^2-\sum_{j\in J_i} X_j(t_2)^2.
\end{align}
Summing over all $i$ yields
\begin{align}\label{vd6}
0 \leq \sumi  X_i(t_1)\bigpar{X_i(t_1)-\YY_i}
\le \sumi X_i(t_1)^2-\sumi\sum_{j\in J_i} X_j(t_2)^2 
=Q(t_1)-Q(t_2).
\end{align}
This implies \eqref{l2q}. Moreover, 
for $a>0$, \eqref{vd6} implies
\begin{align}\label{vd7}
\sum_{i:X_i(t_{1})\ge a} \bigpar{X_i(t_1)-\YY_i}
\le
a\qw\sum_{i:X_i(t_{1})\ge a}  X_i(t_1)\bigpar{X_i(t_1)-\YY_i}
\le
a\qw\bigpar{Q(t_1)-Q(t_2)}.
\end{align}
Furthermore, obviously
\begin{align}\label{vd8}
\sum_{i\le k:X_i(t_{1})< a} \bigpar{X_i(t_1)-\YY_i}
\le
\sum_{i\le k:X_i(t_{1})< a} X_i(t_1)
\le ka.
\end{align}
Consequently, by combining \eqref{vd7} and \eqref{vd8}, for every $a>0$,
\begin{align}\label{vd9}
\sumik \bigpar{X_i(t_1)-\YY_i}
\le
a\qw\bigpar{Q(t_1)-Q(t_2)}+ka.
\end{align}
Optimizing by choosing
$a:=\sqrt{(Q(t_1)-Q(t_2))/k}$ yields
the second inequality in \eqref{l2}.
\end{proof}

\begin{proof}[Proof of \refT{T1}]
Note first that $Q\nn(t)\le Q\nn(0)<\infty$ for every $t$  by \eqref{l2q}.
Let, similarly to \eqref{qsk}, $S_k\nn(t):=\sumik X_i\nn(t)$.
Let stopping times $\tau_n$ and positive $h_n$ be as in the statement.
Then \eqref{qn2p} holds by assumption
and thus \refL{L2} implies that for every fixed $k\ge1$,
\begin{align}\label{kq1}
  S_k\nn(\tau_n) - S_k\nn(\tau_n+h_n)\pto0.
\end{align}
Since $X\nn_i(t)=S\nn_i(t)-S\nn_{i-1}(t)$ (with $S\nn_0(t):=0$),
it follows that
\begin{align}\label{kq2}
  \bigabs{X_i\nn(\tau_n) - X_i\nn(\tau_n+h_n)}\pto0
\end{align}
for every $i$.
The definition \eqref{vd1} implies that for every $k\ge1$
\begin{align}\label{kq3}
  d\bigpar{X\nn(\tau_n),X\nn(\tau_n+h_n)}
\le \sumik 2^{-i}   \bigabs{X_i\nn(\tau_n) - X_i\nn(\tau_n+h_n)}
+2^{-k},
\end{align}
and consequently it follows easily from \eqref{kq2} that
\begin{align}\label{kq4}
  d\bigpar{X\nn(\tau_n),X\nn(\tau_n+h_n)}
\pto0.
\end{align}
This shows that $X\nn(\cdot)$ satisfies Aldous's criterion
\cite{Aldous}.
We use the version in \cite[Theorem 16.11]{Kallenberg},
which is stated for processes with values in
an arbitrary complete separable metric space, and 
together with our assumption \ref{T1i} implies
that the family $X\nn(\cdot)$ is tight in $D(\bbR_+,\SSS)$
by a simple argument using \cite[Theorem A2.2]{Kallenberg}.
(See the proof of \cite[Theorem 16.10]{Kallenberg},
which does this assuming finite-dimensional convergence.)
\end{proof}

\begin{proof}[Proof of \refC{CT1}]
We have $Q\nn(\tau_n)-Q\nn(\tau_n+h_n)\ge0$ by \eqref{l2q} in \refL{L2}.
Thus \eqref{qn2}  is equivalent to 
$\E|Q\nn(\tau_n)-Q\nn(\tau_n+h_n)|\to0$, which implies \eqref{qn2p}.
\end{proof}

\begin{proof}[Proof of \refT{T2}]
The family $X^{(n)}(\cdot)$ is tight in $D(\bbR_+,\SSS)$ by \refT{T1} 
(or \refC{CT1}),
which together with the assumed finite-dimensional convergence implies
convergence in distribution in $D(\bbR_+,\SSS)$, see e.g.\ \cite[Theorem
13.1]{Billingsley2ed}. 
\end{proof}

\section{An application: fragmentation of a tree}
\label{Stree}

Consider now the special case of a fragmentation process of a (deterministic or random) tree, 
where edges are deleted sequentially one by one in random order. 
More precisely, we assume that we are given a random tree
$\cT\nn$ of size (number of vertices) $n$ for each $n\ge1$
and 
we assume that conditioned on the tree, each edge $e$ is deleted at some
random time $T_e$, independent of all other edges.
The results extend to trees with arbitrary sizes $|\cT\nn|\to\infty$
with only notational changes.
We further assume, for convenience, that the random times
$T_e$ are i.i.d.\ with a common $\Exp(r_n)$ distribution, for some given
rates $r_n$. 
(The argument works with a minor modification also for $T_e\sim U(0,t_n)$
with $t_n\to\infty$; 
alternatively this case follows easily from the exponential case studied below,
see \refC{CoroFixI}.)

Let $\cT\nn(t)$ be the fragmented forest at time $t$, so $\cT\nn(0)=\cT\nn$,
and denote its components by $\cT\nn_i(t)$, $i=1,\dots,\nu\nn(t)$,
arranged in order of decreasing size. Here, $\nu\nn(t)$ denotes the number of components at time $t$. 
We define a fragmentation process $X\nn(\cdot)=(X_i\nn(\cdot))\ixoo$
as the normalized component sizes, i.e.,
formally by  
$X\nn_i(t):=|\cT\nn_i(t)|/n$, with $X\nn_i(t):=0$ for $i>\nu\nn(t)$.
Note that $X\nn(t)\in\SSSi$ for all $t\ge0$, with $X\nn(0)=(1,0,0,\dots)$.

\begin{theorem}\label{TT}
  Let\/ $\cT\nn(\cdot)$ and $X\nn(\cdot)$ be as above 
and suppose that, with $Q\nn(t)$ given by \eqref{qn},
\begin{align}\label{uj4}
 \limsup_{\ntoo}\bigpar{1- \E Q\nn(t)}\to0
\qquad\text{as $t\downto 0$}.
\end{align}
Then,  the family of processes
$X\nn(\cdot)$ is tight in $D(\bbR_+,\SSSoi)$.
\end{theorem}

\begin{proof}
When an edge $e$ is deleted at the
time $T_e$, some component $\bft$ of $\cT\nn(T_e-)$ is split into two $\bft'$
and $\bft''$. Thus, assuming as we may that all $T_e$ are distinct, 
\begin{align}\label{uj1}
  Q\nn(T_e-)&-Q\nn(T_e) 
=\frac{1}{n^2}\bigpar{|\bft|^2-|\bft'|^2-|\bft''|^2}
=\frac{2}{n^2}|\bft'||\bft''|
\notag\\&
=
\frac{1}{n^2}\sum_{v,w\in \cT\nn}\indic{\text{$v$ and $w$ are connected in
  $\cT\nn(T_e-)$ but not in $\cT\nn(T_e)$}}
\notag\\&
=
\frac{1}{n^2}\sum_{v,w\in \cT\nn}\etta\bigl\{
\text{$e$ lies in the path between $v$ and $w$ in $\cT\nn(0)$},
\notag\\&\hskip9 em\text{and this path remains  intact in $\cT\nn(T_e-)$}\bigr\}.
\end{align}
Conditioning on $\cT^{(n)}$ and all $T_{e'}$, $e'\neq e$, 
each term in the final sum in  \eqref{uj1} is a decreasing
function of $T_e$, and since $T_e$ 
is independent of the past (unless it already has occured), 
it follows that for any stopping time $\tau_n$
and constant $h_n>0$, 
\begin{multline}
\E\Bigpar{ \bigpar{Q\nn(T_e-)-Q\nn(T_e)}\indic{\tau_n< T_e\le \tau_n+h_n}   }
\\
\le
\E\Bigpar{\bigpar{Q\nn(T_e-)-Q\nn(T_e)}\indic{0< T_e\le h_n}   }.
\end{multline}
Summing over all edges $e$ yields
\begin{align}\label{uj2}
  \E \bigsqpar{Q\nn(\tau_n)-Q\nn(\tau_n+h_n)}
\le
  \E \bigsqpar{Q\nn(0)-Q\nn(h_n)}
=1 - \E Q\nn(h_n).
\end{align}
Hence, the condition \eqref{qn2} reduces to
\begin{align}\label{uj3}
  \E Q\nn(h_n)\to1
\qquad\text{for any sequence $h_n\to0$}.
\end{align}
This follows easily from \eqref{uj4}.
In fact, let $\eps>0$. Then \eqref{uj4} shows that there exists $t_0>0$ such
that $\limsup_{\ntoo}(1-\E Q\nn(t_0))<\eps$.
Hence, for some $n_0$, if $n>n_0$ then $1-\E Q\nn(t_0)<\eps$.
Consequently, if $n$ is so large that $n\ge n_0$ and $h_n<t_0$, then
$\E Q\nn(h_n)\ge \E Q\nn(t_0)>1-\eps$, and thus \eqref{uj3} follows.
(In fact, it is easy to see that 
\eqref{uj3} and \eqref{uj4} are equivalent, and also 
$ \sup_{n\ge 1}\bigpar{1- \E Q\nn(t)}\to0$ as $t\downto 0$,
using that trivially $\cT\nn(t)\to\cT\nn(0)$ 
and thus $Q\nn(t)\to 1$ 
as $t\downto0$,
for every fixed $n$.)

Thus, by Corollary \ref{CT1}, the sequence of processes $X_n(\cdot)$ is tight in
$D(\bbR_+,\SSS)$. This, in turn, implies tightness in $D(\bbR_+,\SSSoi)$
by \refL{LSSSoi}, 
since $X_n(\cdot)\in\SSSoi$.
\end{proof}

\begin{theorem}\label{TT2}
  Let\/ $\cT\nn(\cdot)$ and $X\nn(\cdot)$ be as above 
and suppose that we have finite-dimensional convergence
$X\nn(\cdot)\fdto X(\cdot)$ to some stochastic process $X(\cdot)$ in $D(\bbR_+,\SSS)$.
Then
$X\nn(\cdot)\dto X(\cdot)$ in $D(\bbR_+,\SSSoi)$. 
\end{theorem}

\begin{proof}
For each fixed $t\ge0$ we have by assumption $X\nn(t)\dto X(t)$,
where $X\nn(t)\in\SSSi\subseteq\SSSoi$. Since $\SSSoi$ is a closed subset of
$\SSS$, this implies $X(t)\in\SSSoi$ a.s.\
and
$X\nn(t)\dto X(t)$ in $\SSSoi$.
Moreover, by \refL{L0}, the $\ell^2$ metric is equivalent to the metric
\eqref{vd1} on $\SSSoi$ and thus defines the topology; 
in particular, the $\ell^2$ norm is a continuous function on
$\SSSoi$, and thus we obtain, for every fixed $t\ge0$,
\begin{align}\label{sw4}
  Q\nn(t)=\sumi\bigpar{X\nn_i(t)}^2\dto 
  \sumi\bigpar{X_i(t)}^2
=Q(t).
\end{align}
Since $Q\nn(t)\le1$, we obtain by \eqref{sw4} and dominated convergence, 
\begin{align}\label{sw5}
\E  Q\nn(t)
\to\E Q(t).
\end{align}

By definition, 
$X\nn(0)=(1,0,0,\dots)$ and thus $Q\nn(0)=1$ for every $n$.
Hence \eqref{sw4} implies $Q(0)=1$ a.s.
We assume that $X(\cdot)\in D(\bbR_+,\SSS)$, and in particular $X(t)\to X(0)$
a.s.\ as $t\downto0$.
By the same argument as above, this implies $X(t)\to X(0)$ in $\ell^2$ and
$Q(t)\to Q(0)$ a.s., and thus $\E Q(t)\to\E Q(0)=1$ as $t\downto0$.
Consequently, using \eqref{sw5},
\begin{align}
\lim_{t\to0} \limsup_{\ntoo}\bigpar{1- \E Q\nn(t)}
=
\lim_{t\to0} \bigpar{1- \E Q(t)}
=0,
\end{align}
which shows that \eqref{uj4} holds.
Thus \refT{TT} applies and yields tightness of the family $X^{(n)}(\cdot)$ in
$D(\bbR_+,\SSSoi)$.  
This and the finite-dimensional convergence yield 
convergence in $D(\bbR_+,\SSSoi)$ (as in the proof of \refT{T2}).
\end{proof}

We may also sometimes estimate $Q\nn(t)$ directly.
\begin{theorem}\label{Td}
With notation as in \refT{TT}, 
let $d\nn(v,w)$ denote the graph distance between vertices $v$ and $w$ in
$\cT\nn$. 
Then,  for every $t\ge0$,
\begin{align}
  \label{sof3}
1-\E Q\nn(t)
\le
t\,r_n \E d\nn(V_1,V_2),
\end{align}
where $V_1$ and $V_2$ are two independent uniformly random vertices of\/ 
$\cT\nn$.

Consequently, if 
\begin{align}
  \label{sof4}
r_n \E d\nn(V_1,V_2)=O(1),
\end{align}
then 
the family of processes $X\nn(\cdot)$ is tight in $D(\bbR_+,\SSSoi)$.
\end{theorem}
\begin{proof}
  Fix a time $t\ge0$.
Then, similary to \eqref{uj1},
\begin{align}\label{sof1}
Q\nn(t) &
=
\frac{1}{n^2}\sum_{v,w\in \cT\nn}\indic{\text{$v$ and $w$ are connected in
  $\cT\nn(t)$}}
\notag\\&
=
\frac{1}{n^2}\sum_{v,w\in \cT\nn}\indic{
\text{$T_e>t$ for every $e$ in the path between $v$ and $w$ in $\cT\nn$}},
\end{align}
and thus
\begin{align}\label{sof2}
&1-\E\bigpar{Q\nn(t) \mid \cT\nn}
=\E\bigpar{Q\nn(0)-Q\nn(t) \mid \cT\nn}
\notag\\&\hskip2em
=
\frac{1}{n^2}\sum_{v,w\in \cT\nn}\P\bigpar{
\text{$T_e\le t$ for some $e$ in the path between $v$ and $w$ in $\cT\nn$} \mid \cT\nn} 
\notag\\&\hskip2em
\le
\frac{1}{n^2}\sum_{v,w\in \cT\nn} r_n t d\nn(v,w)
= r_n t \E \bigpar{d\nn(V_1,V_2)\mid\cT\nn},
\end{align} 
which shows \eqref{sof3}.
The final statement follows from \eqref{sof3} and \refT{TT}.
\end{proof}

The following corollary shows that  
we can replace the condition \eqref{sof4} on the average distance
by corresponding conditions on the diameter or (in a rooted tree) 
the total path-length;
it shows  also that 
we then can weaken the condition from bounded in expectation to bounded in
probability. 
If $\cT\nn$ is  a rooted tree, with root denoted by $\rho$,
then its \emph{total path length} is defined as
\begin{align}\label{tpl}
\TPLn:=
\TPL(\cT\nn):=
\sum_{v\in\cT\nn} d\nn(\rho,v)=n\E\bigpar{d\nn(\rho,V)\mid\cT\nn},  
\end{align}
where $V$ is a uniformly random vertex in $\cT\nn$.

For convenience in applications, we state the corollary with
several versions of the conditions
although some are redundant.

\begin{corollary}\label{CD}
Let $\cT\nn$, $r_n$, and $X\nn(\cdot)$ be as above.
Let $D\nn$ be the diameter of $\cT\nn$ and,
provided the tree $\cT\nn$ is rooted,
let $\TPLn$ be its total pathlength.
Suppose that one of the following conditions holds:
\begin{romenumerate}
  
\item \label{CDed}
$r_n\E D\nn=O(1)$.

\item \label{CDod}
$r_nD\nn =\Op(1)$.
(I.e.,  the sequence $r_nD\nn$ is tight.) 
\item \label{CDel}
$r_n\E\TPLn =O(n)$.

\item \label{CDol}
$r_n\TPLn =\Op(n)$.
(I.e.,  the sequence $\frac{r_n}{n}\TPLn$ is tight.) 
\end{romenumerate}
Then,  the family of processes
$X\nn(\cdot)$ is tight in $D(\bbR_+,\SSSoi)$.
\end{corollary}

\begin{proof}
\ref{CDed}: 
Let $V_1$ and $V_2$ be independent random vertices as in \refT{Td}. 
Then $d\nn(V_1,V_2)\le D\nn$, and thus \ref{CDed} implies \eqref{sof4}
and the tightness follows by \refT{Td}.

\ref{CDel}: 
Similarly, if $\cT\nn$ has the root $\rho$, then \eqref{tpl} implies
\begin{align}\label{sof8}
  \E\bigpar{d\nn(V_1,V_2)\mid\cT\nn}
\le
\E\bigpar{d\nn(V_1,\rho)+d\nn(\rho,V_2)\mid\cT\nn}
=\frac{2}{n}\TPLn.
\end{align}
Consequently, 
$\E{d\nn(V_1,V_2)}\le \frac{2}{n}\E\TPLn$, and thus \eqref{sof4} follows
from \ref{CDel}.

\ref{CDod}:
  Let $\eps>0$. By \ref{CDod} there exists a constant $C$ such that
$\P(r_nD\nn>C) < \eps$ for every $n$.
We may apply \ref{CDed} to the conditioned trees 
$(\cT\nn\mid r_nD\nn\le C)$,
and conclude that the conditioned processes 
$(X\nn(\cdot)\mid r_nD\nn\le C)$
form a tight family in $D(\bbR_+,\SSSoi)$.
Hence, there exists a compact set $K$ in $D(\bbR_+,\SSSoi)$
such that, for every $n$,
$\P\bigpar{X\nn(\cdot)\notin K\mid r_nD\nn\le C}\le \eps$,
and consequently
\begin{align}\label{sof9}
\P\bigpar{X\nn(\cdot)\notin K}
\le
\P\bigpar{X\nn(\cdot)\notin K\mid r_nD\nn\le C}
+\P\bigpar{ r_nD\nn> C}
\le 2\eps.
\end{align}
This implies that the family $X^{(n)}(\cdot)$ is tight.

\ref{CDol}: Similar, now using \ref{CDel}.
\end{proof}
Now, let us consider the case where each edge $e$ in $\mathcal{T}^{(n)}$ is
deleted at a random time 
$\hT_e \sim U(0, t_{n})$,
independently for all edges. 
Let
$\widehat{\cT}\nn(t)$ be the fragmented forest at time $t$, and let
$\widehat{X}\nn(\cdot)=(\widehat{X}_i\nn(\cdot))_{i \geq 1}$ be the corresponding
fragmentation process of normalized component sizes, defined analogously as
before in this section.

\begin{corollary} \label{CoroFixI}
With notations as above,
suppose that $t_n\to\infty$ and that
one of the following conditions holds:
\begin{romenumerate}
\item \label{CFId}
$\E d\nn(V_1,V_2)=O(t_{n})$.   
\item \label{CFIed}
$\E D^{(n)}=O(t_{n})$.
\item \label{CFIel}
$\E \TPLn=O(nt_{n})$.
\item \label{CFIod}
$D^{(n)}=\Op(t_{n})$.
\item \label{CFIol}
$\TPLn=\Op(nt_{n})$.
\end{romenumerate}
Then, the family of processes $\widehat{X}\nn(\cdot)$ is tight in
$D(\bbR_+,\SSSoi)$.
\end{corollary}

\begin{proof}
Define $a_{n}(t) := t_{n}(1-e^{-t/t_{n}})$. Observe that
$\widehat{\cT}\nn(a_{n}(t))$ represents the fragmented forest at time $t$
resulting from independently deleting each edge $e$ in $\mathcal{T}^{(n)}$
at a random time $T_e:=-t_n\log(1-\hT_e/t_n)\sim\Exp(r_n)$, 
where $r_n:=t_n\qw$ and the $T_e$'s are i.i.d. 
In particular,
$X\nn(t) = \widehat{X}\nn(a_{n}(t))$. 
If \ref{CFId} holds, then
we have $r_{n}\E d\nn(V_1,V_2)=O(1)$, and thus
\refT{Td} 
implies that the family of processes $X\nn(\cdot)$ is tight in $D(\bbR_+,\SSSoi)$. 
Similarly, if one of the other conditions holds, then the same conclusion
holds by \refC{CD}.

Define $b_{n}(t) := -t_{n} \ln (\max(1-t/t_{n}, 0))$. Observe that
$b_{n}(t)$ converges to the identity function uniformly on each compact
interval $[0,s]$, $s \in \mathbb{R}_{+}$. 
Thus, since $\widehat{X}\nn(t) = X\nn(b_{n}(t))$
for $t< t_n$, 
our claim follows.
\end{proof}

\section{Fixing the proof of tightness in \cite{BOH2023}}
\label{Scorrect}

Let $\mu = (\mu(k), k \geq 0)$ be a probability distribution on the
non-negative integers satisfying $\sum_{k = 0}^{\infty} k \mu(k) = 1$. In
addition, we always implicitly assume that $\mu(0) + \mu(1) < 1$ (or
equivalently, $\mu(0) >0$), and that $\mu$ is aperiodic. We say that $\mu$
belong to the domain of attraction of a stable law of index $\alpha \in
(1,2]$ if for a sequence $(Y_{i})_{i \geq 1}$ of i.i.d.\ random variables
with distribution $\mu$,  there exists a sequence of positive real numbers
$(B_{n})_{n \geq 1}$ such that $B_{n} \rightarrow \infty$ and
\begin{eqnarray}\label{da} 
\frac{Y_{1} + Y_{2} + \cdots + Y_{n} - n}{B_{n}}  \dto Z_{\alpha}, \quad \text{as} \quad n \rightarrow \infty,
\end{eqnarray}
where $Z_{\alpha}$  is a random variable with Laplace transform
$\mathbb{E}[\exp(-\lambda Z_{\alpha})] = \exp(\lambda^{\alpha})$
whenever
$\alpha \in (1,2)$, and $\mathbb{E}[\exp(-\lambda Z_{2})] =
\exp(\lambda^{2}/2)$ if $\alpha=2$, for every  $\lambda > 0$.
(See \cite[Section XVII.5]{FellerII}, and note that  
since $Y_i\ge0$, 
we may normalize $B_n$ such that \eqref{da} holds.)
In particular, for $\alpha =2$,
we have that $Z_{2}$ is distributed as a standard Gaussian random variable.
Furthermore, if \eqref{da} holds, then
$B_{n}$ is of order about $n^{1/\alpha}$; more precisely, $B_{n}/n^{1/\alpha}$
is a slowly varying sequence.

Let $\mathcal{T}^{(n)}_{\rm GW}$ be a critical Galton--Watson tree
conditioned on having $n$ vertices and whose offspring distribution $\mu$
belongs to the domain of attraction of a stable law. Conditioned on the
tree, each edge $e$ is deleted at some random time 
$\hT_e=\frac{n}{B_{n}}(1-U_e)$,
independent of all other edges, 
where as in \cite[Section 1]{BOH2023}, we assume
that the 
$U_e$ are i.i.d.\ with a common $U(0, 1)$ distribution;
hence $\hT_e\sim U(0,n/B_n)$.
Let
$\widehat{X}\nn(\cdot)=(\widehat{X}_i\nn(\cdot))_{i \geq 1}$ be the associated
fragmentation process of normalized component sizes, as defined before
Corollary \ref{CoroFixI}. In the notation of \cite[(3)]{BOH2023},
$\widehat{X}\nn(\cdot) = \mathbf{F}_{n}^{(\alpha)}(\cdot)$.

\begin{corollary} \label{CoroFixII}
The family of processes $\widehat{X}\nn(\cdot)$ is tight in $D(\bbR_+,\SSSoi)$. 
\end{corollary}

\begin{proof}
Let $H^{(n)}$ be the height of $\mathcal{T}^{(n)}_{\rm GW}$. 
Then, $D^{(n)} \leq 2 H^{(n)}$, where $D^{(n)} $ is the diameter of
$\mathcal{T}^{(n)}_{\rm GW}$. It follows from \cite[Theorem 2]{Kortchemski}
that $\mathbb{E}[D^{(n)} ] =
O(n/B_{n})$. Therefore, our claim follows from 
Corollary \ref{CoroFixI}\ref{CFIed}
(with $t_n=n/B_n$).
Alternatively, it follows in the same way from
\refC{CoroFixI}\ref{CFIod}
that is suffices to show that
$(B_n/n)H\nn = \Op(1)$, which follows from the 
convergence in distribution of the height process 
(or of the contour process)
\cite[Theorem 3.1]{Duquesne}, see also \cite[Theorem 3]{Kortchemski2013}.
Or from \refC{CoroFixI}\ref{CFIol}, since the total path length $\TPLn$ of
$\mathcal{T}^{(n)}$ satisfies $\TPLn=\Op(n^{2}/B_{n})$, by \cite[Corollary
4.11]{Delmas2018}.
\end{proof}

\begin{remark}\label{RBOH2023}
Corollary \ref{CoroFixII} addresses a gap in the tightness argument of \cite[Theorem 1]{BOH2023}. Specifically, this issue originated from \cite[Lemma 5]{BOH2023}, which is used in the proof of Theorem \cite[Theorem 1]{BOH2023}. While the convergence of the finite-dimensional distributions in \cite[Lemma 5]{BOH2023} (and consequently in Theorem \cite[Theorem 1]{BOH2023}) is correct, the tightness proof within \cite[Lemma 5]{BOH2023} contains a gap due to the application of the unfortunately incorrect \cite[Lemma 22]{BroutinMarckert}.
We may now instead use \refC{CoroFixII}.

Note that \cite{BOH2023} considers the space $\SSSoo$ defined in \eqref{SSSoo}.
However, in this particular case, since
$\widehat{X}\nn(t)\in\SSSi$ 
for every
$t\ge0$
(a.s.)
and \cite{BOH2023} shows the existence of a finite-dimensional
limit $\widehat{X}(\cdot)\in\SSSi$, 
tightness in $D(\bbR_+,\SSSoi)$ implies 
convergence $\hX_n(\cdot) \dto \hX(\cdot)$ 
in $D(\bbR_+,\SSSoi)$,
which by \refL{LSSSi} and \refR{RSSSoo} implies convergence in
$D(\bbR_+,\SSSi)$ and in
$D(\bbR_+,\SSS_{<\infty})$.
\end{remark}

\section{Fixing the proof of \cite[(2) of Theorem 3]{BroutinMarckert}}
\label{Swrong}

As we said in the introduction, one motivation for the present note is that the compactness result \cite[Lemma 22]{BroutinMarckert} is incorrect. The lemma concerns subsets of the space $D(I,\ell^p_{\ge0})$  of \cadlag{} functions from an interval $I$ to $\ell^p_{\ge0}$, equipped with the Skorohod $J_1$ topology. However, the issue remains for functions with values in a finite-dimensional space, as seen by the following counterexample.

\begin{example}
  For $n\ge2$,
let $g_n$ be the distribution function of a uniform distribution on 
$[\frac12-\frac1n,\frac12]$, i.e., 
$g_n(\frac12-\frac1n)=0$,
$g_n(\frac12)=1$ with $g_n$ linear in between and constant outside  
this interval.
Let $f_n(x)$ be the vector $(g_n(x), 1-g_n(x), 0,0,\dots)$.
It is then easy to see that the set $\set{f_n,n\ge2}$
satisfies the assumptions of \cite[Lemma 22]{BroutinMarckert}
(with $I=[0,1]$, $p=1$, and $M=2$ for every $\eps$).
Nevertheless, the set is not relatively compact since no subsequence of $(f_n)$
converges. To see this it suffices to consider the first components
$g_n(x)$; it is easy to see that 
no subsequence can converge since the set of continuous functions is closed
in $D(I,\bbR)$, and on this subset, convergence in the Skorohod topology is
equivalent to uniform convergence.
The error in the proof 
in \cite{BroutinMarckert} 
is the assertion that any sequence of 
[uniformly] bounded non-decreasing functions on $I$ has an accumulation
point in $D(I,\bbR)$. 
This is correct for the weaker $M_1$ topology, 
see e.g.\ \cite[Corollary 12.5.1]{Whitt},
but not for the standard $J_1$ topology;
a counterexample is provided by $g_n$ above.
\end{example}

We now provide a corrected proof for the claim made in \cite[(2) of Theorem 3]{BroutinMarckert}. Let $\mathcal{T}^{(n)}_{\rm Cayley}$ be a uniform Cayley tree with $n$ vertices (uniformly labelled tree on $\{1, \dots, n\}$).  Conditioned on the tree, each edge $e$ is deleted at some random time $\hT_e=\sqrt{n}(1-U_e)$, independent of all other edges. As in \cite[Section 3.2.3]{BroutinMarckert}, we assume that the $U_e$ are i.i.d.\ with a common $U(0, 1)$ distribution. Let $\widehat{X}\nn(\cdot)=(\widehat{X}_i\nn(\cdot))_{i \geq 1}$ be the associated fragmentation process of normalized component sizes, as defined before Corollary \ref{CoroFixI}. 

In \cite{BroutinMarckert}, the process $\widehat{X}\nn(\cdot)=(\widehat{X}_i\nn(\cdot))_{i \geq 1}$ is encoded using a process $\mathbf{y}_{+}^{n, (t)} = (\mathbf{y}_{+}^{n, (t)}(x))_{x \in [0,1]}$,  as defined in \cite[(7)]{BroutinMarckert}. More precisely, $\widehat{X}\nn(t)$ is the sequence of excursion lengths of $\mathbf{y}_{+}^{n, (t)}$ above its minimum sorted in decreasing order at time $t$. Let $\mathbf{y}_{+}^{(t)}(x) := \mathbf{e}(x)-tx$, for $x \in [0,1]$ and $t \geq 0$, where $\mathbf{e}$ is the normalised Brownian excursion (with unit length). Let $\widehat{X}(t)$ be the sequence of excursion lengths of $\mathbf{y}_{+}^{(t)}$ above its minimum sorted in decreasing order. Following the notation of \cite{BOH2023}, $\widehat{X}\nn(t) = \pmb{\gamma}^{+,n}(t)$ and $\widehat{X}(t) = \pmb{\gamma}^{+}(t)$. 
\begin{theorem}
$\widehat{X}\nn(\cdot)\dto \widehat{X}(\cdot)$ in $D(\bbR_+,\SSSi)$.
\end{theorem}
\begin{proof}
Define $\overline{\mathbf{y}}_{+}^{n, (t)}(x) := \sup_{s \in [0,x]} (-\mathbf{y}_{+}^{n, (t)}(s))$ and $\overline{\mathbf{y}}_{+}^{(t)}(x) := \sup_{s \in [0,x]}(- \mathbf{y}_{+}^{(t)}(s))$, for $x \in [0,1]$. Following \cite[Section 3.1]{Bertoin}, given an increasing path $g$ in $D([0,1], \mathbb{R})$, we write $\mathbf{F}(g)$ for the sequence of the lengths of the intervals components of the complement of the support of the Stieltjes measure $\mathrm{d} g$, arranged in the decreasing order. In particular, 
$\widehat{X}\nn(t) = \mathbf{F}(\overline{\mathbf{y}}_{+}^{n, (t)})$ and $\widehat{X}(t) = \mathbf{F}(\overline{\mathbf{y}}_{+}^{(t)})$. 
Then, using \cite[Theorem 10]{BroutinMarckert} together with \cite[Lemma 4
and Lemma 7]{Bertoin}, one can show the finite-dimensional convergence
$\widehat{X}\nn(\cdot)\fdto \widehat{X}(\cdot)$ in $D(\bbR_+,\SSSi)$. 

On the other hand, 
Corollary \ref{CoroFixI} (with $t_n=\sqrt n$), 
shows that the family of processes
$\widehat{X}\nn(\cdot)$ is tight in $D(\bbR_+,\SSSoi)$,
because
it was shown by \cite{Renyi1967} that  
the diameter $D^{(n)}$ of $\mathcal{T}^{(n)}$ satisfies
$\mathbb{E}[D^{(n)}] = O(\sqrt{n})$,
see also \cite[Corollary 1.3]{Svante}.
Alternatively, we can use that
$D\nn=\Op(\sqrt n)$ by \cite{AldousII,AldousIII},
or that
the total path length $\TPLn$ of $\mathcal{T}^{(n)}$ satisfies
$\mathbb{E}[\TPLn]=O(n^{3/2})$
by \cite{Takacs1992,Takacs1993},
see also \cite[Theorem 3.4]{Svante2003};
furthermore, we
may also use 
$\E d\nn(V_1,V_2)=O(n^{1/2})$
which was proved in 
\cite{MeirMoon1970};
again see also \cite[Theorem 3.4]{Svante2003}.
Combining this and finite-dimensional convergence,
we obtain  $\widehat{X}\nn(\cdot)\dto \widehat{X}(\cdot)$ in $D(\bbR_+,\SSSoi)$.
Finally, since
$\widehat{X}(t)\in\SSSi$ for every $t$ a.s.\ (see \cite[Lemma 7]{Bertoin}),
our claim follows by \refL{LSSSi}.
\end{proof}

\section{Application to the fragmentation of trees with a specified degree sequence} \label{Sfixed}

For $n \in \mathbb{N}$, a sequence $\mathbf{s}_{n} = (N_{n}(i))_{i \geq 0}$
of non-negative integers is the degree sequence of some finite rooted plane
tree if and only if $\sum_{i \geq 0} N_{n}(i)= 1 + \sum_{i \geq 0} i
N_{n}(i) < \infty$. Then, a random tree $\mathcal{T}^{(n)}_{\mathbf{s}}$
with given degree sequence $\mathbf{s}_{n}$ is a random variable whose law
is uniform on the set of rooted plane trees with $|\mathbf{s}_{n}| := \sum_{i \geq 0}
N_{n}(i)$ vertices amongst which $N_{n}(i)$ have $i$ offspring for every $i
\geq 0$, and therefore with
$\sum_{i \geq 0} i N_{n}(i)$ edges. 

Define $\sigma^2_{n} := \sum_{i \geq 1} i(i-1)N_{n}(i)$ and let $b_{n}$ be a sequence of non-negative real number such that $b_{n} \rightarrow \infty$, $|\mathbf{s}_{n}|/b_{n} \rightarrow \infty$ and $b_{n}/\sigma_{n} \rightarrow 1$. Conditioned on the tree, each edge $e$ is deleted at some random time $\hT_e=\frac{|\mathbf{s}_{n}|}{b_{n}}(1-U_e)$, independent of all other edges. We assume that the $U_e$ are i.i.d.\ with a uniform $U(0, 1)$ distribution. Let $\widehat{X}\nn(\cdot)=(\widehat{X}_i\nn(\cdot))_{i \geq 1}$ denote the associated fragmentation process of normalized component sizes, as defined prior to Corollary \ref{CoroFixI}.

\begin{corollary} \label{CoroFixIII}
The family of processes $\widehat{X}\nn(\cdot)$ is tight in $D(\bbR_+,\SSSoi)$. 
\end{corollary}

\begin{proof}
For a vertex $v$ in  $\mathcal{T}^{(n)}_{\mathbf{s}}$, we denote its height
by $h\nn(v)$. Without loss of generality, by considering $n$ large enough if
necessary, we assume that 
$|\mathbf{s}_{n}| \geq b_{n}$ and $\sigma_{n} >0$.
Let $\Delta_{n} := \max\{i \geq 0: N_{n}(i)>0 \}$. 
Then, since $\sigma_{n} >0$, we have
$\Delta_{n} \geq 2$ and 
$\sigma_n \ge \Delta_{n}$.
Let $V_1$ and $V_2$ be two independent uniformly random
vertices of $\mathcal{T}^{(n)}_{\mathbf{s}}$. Observe that $d^{(n)}(V_{1},
V_{2}) \leq h^{(n)}(V_{1}) + h^{(n)}(V_{2})$. It follows from
\cite[Proposition 4.5]{Marzouk} 
that there exists two universal constants
$c_{1}, c_{2}>0$ such that  
\begin{align}
\mathbb{P}\bigpar{h^{(n)}(V_{1}) \geq x |\mathbf{s}_{n}|/b_{n}  } \leq c_{1}e^{-c_{2}x \frac{\sigma_{n}}{b_{n}}}
\end{align}
\noindent uniformly for $x >0$ 
and $n \in \mathbb{N}$. 
(See also \cite[Theorem 1.1]{Louigi} for a similar bound.)
Thus, $\mathbb{E}d^{(n)}(V_{1}, V_{2}) \le 2\E h\nn(V_1)= O(|\mathbf{s}_{n}|/b_{n})$ and
our claim follows from Corollary \ref{CoroFixI}\ref{CFId}.
\end{proof}

\begin{remark}\label{RBHP2025} 
Recently, the fragmentation process of trees with a given degree sequence
has been studied in \cite{BHP2025}, where the convergence of
$\widehat{X}\nn(\cdot)=(\widehat{X}_i\nn(\cdot))_{i \geq 1}$ was established. However, at the time of writing this note, a problem existed in the tightness argument, specifically originating from the unfortunately incorrect \cite[Lemma 22]{BroutinMarckert}. Now, Corollary \ref{CoroFixIII} completes the proof. 
As in \refR{RBOH2023}, the result holds also in 
$D(\bbR_+,\SSSi)$ and $D(\bbR_+,\SSS_{<\infty})$
by \refL{LSSSi} and \refR{RSSSoo}.
\end{remark}

\section{Fragmentations of a weighted tree}\label{Sweight}
In \refSs{Stree}--\ref{Sfixed} we have studied fragmentations of trees,
where we have measured the sizes of components by their number of vertices.
More generally, we may assume that we are given, for each $n$,  
a random tree $\cT\nn$ with a nonrandom vertex set
$\cV\nn=\cV(\cT\nn)$ (finite or infinite), and
a family of nonnegative weights $\bw\nn=(w\nn_j:j\in \cV\nn)$ with total mass
$\sum_j w\nn_j<\infty$.
We consider random fragmentation of the tree $\cT\nn$ as above, with edges
$e$ deleted at i.i.d.\ times $T_e$ that are either $\Exp(r_n)$ or $U(0,t_n)$.
We then define the fragmentation process $Y\nn(\cdot)=(Y_i\nn(\cdot))\ixoo$ by letting
\begin{align}
  Y\nn_i(t):=\sum_{j\in \cT\nn_i(t)} w_j\nn,
\end{align}
the total weight of the component $\cT\nn_i(t)$, for $i=1,\dots,\nu\nn(t)$, where we now arrange the
components in order of decreasing weight.

For simplicity we will only consider the case when the weight sequence 
$\bw\nn$ is a probability distribution on $\cV\nn$, i.e.,
$\sum_j w\nn_j=1$. To stress this we use the notation $\bp\nn=(p\nn_j:j \in \cV\nn)$
instead of $\bw\nn$. 
(We leave the modifications for the general case to the reader.)

\begin{theorem}
  \label{TTW}
With assumptions as above, 
\refTs{TT} and \ref{TT2} still hold (with $Y\nn$ and $Y$ instead of $X\nn$ and $X$), and \refT{Td} holds if $V_1$ and $V_2$
now are independent vertices of $\cT\nn$ with the distribution $\bp\nn$.
\end{theorem}
\begin{proof}
  The proofs are essentially the same.

First, for \refT{TT},
instead of \eqref{uj1} we obtain, in the same way,
\begin{multline}\label{uj1w}
  Q\nn(T_e-)-Q\nn(T_e) 
=
\sum_{v,w\in \cV\nn}p\nn_vp\nn_w\etta\bigl\{
\text{$e$ lies in the path between $v$ and $w$ in}
\\\text{$\cT\nn(0)$,
and this path remains  intact in $\cT\nn(T_e-)$}\bigr\}.
\end{multline}
The rest of the proof is the same as before.

The proof of \refT{TT2} is unchanged.

In the proof of \refT{Td}, we obtain, arguing as in \eqref{sof1}--\eqref{sof2},
\begin{align}\label{sof2w}
&1-\E\bigpar{Q\nn(t) \mid \cT\nn}
\le
\sum_{v,w\in \cV\nn} p_v^{(n)}p_w^{(n)}r_n t d\nn(v,w)
= r_n t \E \bigpar{d\nn(V_1,V_2)\mid\cT\nn},
\end{align}
which shows \eqref{sof3}.
\end{proof}

\begin{remark} \label{RemarkExtI}
The parts of \refCs{CD} and \ref{CoroFixI} that use $D\nn$ still hold. 
However, the parts involving the total path length $\gU\nn$ require us to instead consider the weighted path length $\sum_{v\in V\nn}p\nn_vd\nn(\rho,v)$. Equivalently, these parts can be stated in terms of the average depth
$\E d\nn(\rho,V)$, where $V$ is a random vertex with distribution $\bp\nn$. The part involving $\E d\nn(V_1,V_2)$ also holds, but we must consider two independent random vertices $V_{1}$ and $V_{2}$  of $\mathcal{T}^{(n)}$ chosen according to the distribution $\mathbf{p}\nn$. The details are left to the reader.
\end{remark}

\section{Application to the fragmentation of $\mathbf{p}$-trees} \label{ptreesApp}

For $n \geq 1$, let $\mathbf{p}\nn = (p_{ni}: i \geq 1)$ be a ranked
probability distribution on the positive integers. That is, $p_{n1}
\geq p_{n2} \geq 0$ and $\sum_{i \geq 1} p_{ni} =1$. The support of
$\mathbf{p}\nn$ is $\mathcal{V}\nn :=\{i \geq 1: p_{ni}>0 \}$. Let
$Y_{n,0}, Y_{n,1}, \dots$  be i.i.d.\ random variables with common
distribution $\mathbf{p}\nn$. Define a random discrete tree
$\mathcal{T}_{\mathbf{p}}^{(n)}$ to have vertex set $\mathcal{V}\nn$ and
undirected edges $\{ \{Y_{n,j-1}, Y_{n,j} \}: Y_{n,j} \not \in \{Y_{n,0},
\dots, Y_{n,j-1} \}, \, j \geq 1 \}$. The tree
$\mathcal{T}_{\mathbf{p}}^{(n)}$ is known as the birthday tree with
parameter $\mathbf{p}\nn$, or simply as the $\mathbf{p}\nn$-tree (see,
e.g., \cite{Camarri}, \cite[Section 3.1]{Aldous2000},
\cite[Section 10.2]{Pitman}, 
and references
therein). Observe that when $p_{ni}=1/n$, for $1 \leq i \leq n$, the
$\mathbf{p}\nn$-tree is a uniform random rooted tree with $n$ vertices (the
Cayley tree of size $n$), discussed above in \refS{Swrong}.

Define $\sigma_{\mathbf{p}\nn} := \sqrt{\sum_{i \geq 1}  p_{ni}^{2}}$. 
Following \cite[Section 3.1]{Aldous2000},
conditioned on
$\mathcal{T}_{\mathbf{p}}^{(n)}$, we delete each edge $e$ at some random
time  
$\hT_e=\sigma_{\mathbf{p}\nn}^{-1}(1-U_e)$, independent of all other edges, 
where we assume that the $U_e$ are i.i.d.\ with a common $U(0, 1)$
distribution. 
We define as in \refS{Sweight}
the fragmentation process $\widehat{Y}\nn(t) =
(\widehat{Y}_i\nn(t))_{i \geq 1}$ as the sequence of
$\mathbf{p}\nn$-measures of the tree-components of the fragmented forest at
time $t$, arranged in decreasing order.  
So,
$\widehat{Y}\nn(t)$ records the $\mathbf{p}\nn$-measures of components
obtained when each edge is deleted with probability
$t\sigma_{\mathbf{p}\nn}$.

\begin{corollary} \label{CorPtrees}
Suppose that $\sigma_{\mathbf{p}\nn} \rightarrow 0$. Then, the family of processes $\widehat{Y}\nn(\cdot)$ is tight in $D(\bbR_+,\SSSoi)$. 
\end{corollary}

\begin{proof}
This follows from the weighted version of \refC{CoroFixI} (see Remark
\ref{RemarkExtI}) provided that $\mathbb{E}d^{(n)}(V_{1},
V_{2}) = O(\sigma_{\mathbf{p}\nn}^{-1})$, where $V_{1}$ and $V_{2}$ are two
independent random vertices of $\mathcal{T}_{\mathbf{p}}^{(n)}$ chosen
according to the distribution $\mathbf{p}\nn$. But this last claim follows
from Proposition \ref{ProTailDist}
in \refApp{App},
thus concluding our proof. 
\end{proof}

\begin{remark}
Define 
\begin{align}
\SSS_{2} := \Bigset{(x_i)\xoo\in\SSS: 0 \leq \sumi x_i^{2} \leq 1 }
\end{align}
\noindent and
\begin{align}
\Theta := \Bigset{(x_i)\xoo\in\SSS_{2}: \sumi x_i^{2} <1 \, \, \text{or} \, \,  \sumi x_i = \infty}.
\end{align}
\noindent Aldous and Pitman \cite[Proposition 13]{Aldous2000} 
proved that if there exists a sequence $\boldsymbol{\theta} = (\theta_i)\xoo \in \Theta$ such that
\begin{align}
\sigma_{\mathbf{p}\nn} \rightarrow 0 \quad \text{and} \quad \frac{p_{ni}}{\sigma_{\mathbf{p}\nn}} \rightarrow \theta_{i}, \quad i \geq 1,
\end{align}
then, $\widehat{Y}\nn(\cdot)\fdto \widehat{Y}(\cdot)$ to some stochastic process $\widehat{Y}(\cdot)$ in $D(\bbR_+,\SSSi)$. On the other hand, we have that $\widehat{Y}\nn(t)\in\SSSi$ and $\widehat{Y}(t)\in\SSSi$ for every
$t\ge0$ (a.s.) (\cite[Lemma 12]{Aldous2000}). Then, tightness in
$D(\bbR_+,\SSSoi)$ (Corollary \ref{CorPtrees}) implies 
$\widehat{Y}\nn(\cdot)\dto \widehat{Y}(\cdot)$ in $D(\bbR_+,\SSSi)$
by \refL{LSSSi}.
The limit $\widehat{Y}(\cdot)$ is the fragmentation process of the so-called 
\emph{Inhomogeneous continuum random tree} with parameter $\boldsymbol{\theta}$; see \cite[Section 4]{Aldous2000} for its definition. 
\end{remark}

\begin{remark}
Define $M\nn(t) := \widehat{Y}\nn(\sigma_{\mathbf{p}_{n}}^{-1}e^{-t})$, for $0 \leq t < \infty$. Then, $M\nn(\cdot)$ is an additive coalescent with initial state $\mathbf{p}_{n}$; see \cite[Proposition 1]{Aldous2000}. Indeed, if we define $M(t) := \widehat{Y}(e^{-t})$, for $-\infty < t < \infty$, then $M(\cdot)$ is an ({\sl eternal}) additive coalescent; see \cite[Theorem 10]{Aldous2000}.

In \cite{Bertoin}, a specific case is analysed where $p_{nn} > 0$ and $p_{ni} = 0$ for $i > n$. \cite[Theorem 1]{Bertoin} proves that the process $(M\nn(-\log \sigma_{\mathbf{p}_{n}} - \log t), 0 < t < \sigma_{\mathbf{p}_{n}}^{-1})$ converges in the sense of finite dimensional distributions, as $n \rightarrow \infty$, in $D(\bbR_+,\SSSi)$, to a fragmentation process associated with certain bridges with exchangeable increments. Observe that $M\nn(-\log \sigma_{\mathbf{p}_{n}} - \log t) = \widehat{Y}\nn(t)$. Then, Corollary \ref{CorPtrees} and \refL{LSSSi} imply the tightness of this sequence of processes in $D(\bbR_+,\SSSi)$. As a consequence, the convergence in \cite[Theorem 1]{Bertoin} can be strengthened to convergence in distribution in $D(\bbR_+,\SSSi)$. 
\end{remark}

\appendix
\section{An exponential tail bound} \label{App}

In this appendix, we prove an exponential tail bound, used above,
on the distance between two independent random vertices of a  $\mathbf{p}_{n}$-tree, chosen independently according to the distribution $\mathbf{p}_{n}$. For notational simplicity, we will drop the $n$ dependence from the notation introduced in Section \ref{ptreesApp} and write $\mathbf{p} = (p_{i}: i \geq 1)$ for a ranked  probability distribution on the positive integers, and $\mathcal{T}_{\mathbf{p}}$ for the associated $\mathbf{p}$-tree.

\begin{proposition} \label{ProTailDist}
Let\/ $V_{1}$ and $V_{2}$ be two independent random vertices of $\mathcal{T}_{\mathbf{p}}$ chosen according to the distribution $\mathbf{p}$. Then, for all $x \geq 8$,
\begin{align}
\mathbb{P}(d(V_{1}, V_{2}) \geq x \sigma_{\mathbf{p}}^{-1} ) \leq e^{-\frac{x^{1/3}}{3\sigma_{\mathbf{p}}}} + 6e^{-\frac{x^{2/3}}{6}}. 
\end{align} 
\end{proposition}

The proof of Proposition \ref{ProTailDist} is based on the Poisson embedding described in \cite[Section 4.1]{Camarri}. Let $N$ be a homogeneous Poisson process on $[0, \infty) \times[0,1]$ of rate $1$ per unit area, with atoms $((S_{j}, U_{j}))_{j \geq 0}$ such that $0<S_{0}<S_{1}<\cdots$. Then, $(S_{j})_{j \geq 0}$ are the atoms of a homogeneous Poisson process on $[0, \infty)$ of rate $1$ per unit length, and the $(U_{j})_{j \geq 0}$ are i.i.d.\ with uniform distribution on $[0,1]$, independent of the $(S_{j})_{j \geq 0}$. Define
\begin{align}
N(t):=N([0, t] \times[0,1]) \quad \text { and } \quad N(t-):=N([0, t) \times[0,1]).
\end{align}
\noindent Then, partition the interval $[0,1]$ into sub-intervals $I_{1}, I_{2}, \ldots$ such that the length of $I_{i}$ is $p_{i}$, for $i \geq 0$. For $j \geq 0$ define
\begin{align}
Y_{j}:=\sum_{i \geq 1} i \mathbf{1}_{\{U_{j} \in I_{i}\}}.
\end{align}
\noindent So, $Y_{0}, Y_{1}, \dots$ are i.i.d.\ random variables with common distribution $\mathbf{p}$. 

From now on, assume that the $\mathbf{p}$-tree $\mathcal{T}_{\mathbf{p}}$ is constructed using the random variables defined above; see Section \ref{ptreesApp}. Let $R_{1}$ be the index of the first repeated value in the sequence $(Y_{j})_{j \geq 0}$ and let $T_{1 }:=\inf \{t \geq 0: N(t)>R_{1}\}$. Then,
as a consequence of \cite[Corollary 3]{Camarri},
\begin{align} \label{IdenDisI}
N(T_{1}-) = R_{1} \eqd d(V_{1}, V_{2})+1,
\end{align}
\noindent where $V_{1}$ and $V_{2}$ are two independent random vertices of
$\mathcal{T}_{\mathbf{p}}$ chosen according to the distribution $\mathbf{p}$.

The proof of Proposition \ref{ProTailDist} makes use of the following result, which is a straightforward consequence of \cite[(26) and (29) in Lemma 9]{Camarri}. 
\begin{lemma} \label{LemmaTail}
For $0 \leq t < (2p_{1})^{-1}$, 
we have $\mathbb{P}(T_{1} > t) \leq e^{-\frac{t^{2}}{6}\sigma_{\mathbf{p}}^{2} }$.
\end{lemma}

\begin{proof}
Note that the right-hand side of \cite[(29) in Lemma 9]{Camarri} (with
$\boldsymbol{\theta}=\mathbf{p}$) can be written as
$(t\sigma_{\mathbf{p}})^{2} \frac{tp_{1}}{3(1-tp_{1})} \leq
\frac{t^{2}}{3}\sigma_{\mathbf{p}}^{2}$ because $0 \leq t <
(2p_{1})^{-1}$. Thus, our claim follows from \cite[(26) and (29)]{Camarri}.
\end{proof}

\begin{proof}[Proof of Proposition \ref{ProTailDist}]
It follows from \eqref{IdenDisI} that, for 
$h \ge 0$ and 
$0\le t\le h$,
\begin{align} \label{eq1Tail}
\mathbb{P}(d(V_{1}, V_{2}) \ge h) \leq  \mathbb{P}(N(t) \geq h) + \mathbb{P}(T_{1} > t).
\end{align}
\noindent On the other hand, recalling the following Chernoff-type
inequality (see, e.g., 
\cite[Corollary 2.4 and Remark 2.6]{JLR}
or
\cite[Section 2.2, particularly page 23]{Boucheron}),
we have that, for $t \leq h$,
\begin{align} \label{eq2Tail}
\mathbb{P}(N(t) \geq h) \leq e^{-t ((h/t) \log(h/t)-h/t+1)}.
\end{align}
\noindent In particular, if $h = 2t$, we have that $t ((h/t) \log(h/t)-h/t+1) = t (2 \log 2-1) \ge t/3$ and thus
\begin{align} \label{eq3Tail}
\mathbb{P}(N(t) \geq 2t) \leq e^{-\frac{t}{3}}.
\end{align}

Fix $c \geq 2$. If $0 \leq c\sigma_{\mathbf{p}}^{-1}  < (2p_{1})^{-1}$, then taking $t = c\sigma_{\mathbf{p}}^{-1}$ and $h =2t$, by Lemma \ref{LemmaTail}, \eqref{eq1Tail} and \eqref{eq3Tail} we have that
\begin{align} \label{eq4Tail}
\mathbb{P}(d(V_{1}, V_{2}) > 2c\sigma_{\mathbf{p}}^{-1}  )  \leq  e^{-\frac{c}{3\sigma_{\mathbf{p}}}} + e^{-\frac{c^{2}}{6}}.
\end{align}

Now assume that $c\sigma_{\mathbf{p}}^{-1} \geq (2p_{1})^{-1}$. For any positive real $k\geq 2$, if at least two of the variables $U_{0}, \dots, U_{\floor{k}}$ lie in the interval $I_{1}$ then $N(T_{1}-) \leq k+1$. Thus, by \eqref{IdenDisI},
\begin{align} \label{eq5Tail}
\mathbb{P}(d(V_{1}, V_{2}) \ge k) & \leq \mathbb{P}(\# \{ j \in \{0, \dots, \floor{k} \}: U_{j} \in I_{1} \} \leq 1) \nonumber \\
& = (1-p_{1})^{\floor{k}} (1-p_{1}+ (\floor{k}+1)p_{1}) \nonumber \\
& \leq e^{-p_{1}\floor{k}}(1+\floor{k}p_{1}),
\end{align}
\noindent where we have used that $\# \{ j \in \{0, \dots, \floor{k} \}: U_{j} \in I_{1} \} \sim \mathrm{Binomial}(\floor{k}+1, p_{1})$ and that $1-p_{1} \leq e^{-p_{1}}$. Since $\sigma_{\mathbf{p}} \leq 1$, it follows that for $k \geq 2c\sigma_{\mathbf{p}}^{-1} \geq 4$,
\begin{align} \label{eq6Tail}
\mathbb{P}(d(V_{1}, V_{2}) \ge k) & \leq  2kp_{1}e^{-\frac{kp_{1}}{2}},
\end{align}
\noindent where we have used the lower bound on $k$ to deduce that
$\floor{k}> k/2$ and that  
$kp_1 \geq 2p_1c\sigma_{\mathbf{p}}^{-1} \geq 1$
and so $1+\floor{k}p_{1}\leq 2kp_{1}$.  

Next, take $k = yc\sigma_{\mathbf{p}}^{-1}$, for $y \geq 4$. Recall that we
have assumed that $c\sigma_{\mathbf{p}}^{-1} \geq (2p_{1})^{-1}$ and thus,
$kp_{1} \geq y/2$. Since $z \mapsto 2ze^{-z/2}$ is decreasing for $z \geq
2$, the bound \eqref{eq6Tail} then implies that
\begin{align} \label{eq7Tail}
\mathbb{P}(d(V_{1}, V_{2}) \ge yc\sigma_{\mathbf{p}}^{-1}) & \leq  ye^{-\frac{y}{4}}.
\end{align}

To conclude the proof, we combine \eqref{eq4Tail} and \eqref{eq7Tail} to obtain a bound that does not depend on the value of $p_{1}$. Take $x \geq 8$, let $c=x^{1/3}\geq 2$ and $y = x^{2/3}\geq 4$. Then, $2c \leq x$ and $yc=x$. Thus, regardless of the value of $p_{1}$, either \eqref{eq4Tail} or \eqref{eq7Tail} applies, so we obtain that
\begin{align} \label{eq8Tail}
\mathbb{P}(d(V_{1}, V_{2}) \ge x\sigma_{\mathbf{p}}^{-1}) & =  \mathbb{P}(d(V_{1}, V_{2}) \ge yc\sigma_{\mathbf{p}}^{-1}) \nonumber \\
& \leq  e^{-\frac{c}{3\sigma_{\mathbf{p}}}} + e^{-\frac{c^{2}}{6}} + ye^{-\frac{y}{4}} \nonumber \\
& = e^{-\frac{x^{1/3}}{3\sigma_{\mathbf{p}}}} + e^{-\frac{x^{2/3}}{6}} + x^{2/3}e^{-\frac{x^{2/3}}{4}}.
\end{align}
\noindent Finally, our claim follows from \eqref{eq8Tail} and the
observation that  
$ze^{-z/4} \leq 5 e^{-z/6}$
for all $z \geq 0$. 
\end{proof}

\begin{acks}
We thank an anonymous referee for finding a serious gap in the statement and
proof of the main result.
\end{acks}

\newcommand\AMS{Amer. Math. Soc.}
\newcommand\Springer{Springer-Verlag}
\newcommand\Wiley{John Wiley \& Sons}

\newcommand\vol{\textbf}
\newcommand\jour{\emph}
\newcommand\book{\emph}
\newcommand\inbook{\emph}
\def\no#1#2,{\unskip#2, no. #1,} 
\newcommand\toappear{\unskip, to appear}

\newcommand\arxiv[1]{\texttt{arXiv}:#1}
\newcommand\arXiv{\arxiv}

\newcommand\xand{and }
\renewcommand\xand{\& }

\def\nobibitem#1\par{}


\begin{thebibliography}{99}


\bibitem{Louigi}
Louigi Addario-Berry, Anna Brandenberger, Jad Hamdan, C\'elin Kerriou.
Universal height and width bounds for random trees.
\emph{Electron. J. Probab.} \vol{27} (2022), Paper No. 118, 24 pp.
\MR{4479914}

\bibitem{Svante}
Louigi Addario-Berry, Luc Devroye and Svante Janson.
Sub-{G}aussian tail bounds for the width and height of conditioned {G}alton--{W}atson trees. 
\emph{Ann. Probability} \vol{41} (2013),  1072--1087.
\MR{3077536}

\bibitem{Aldous}
David Aldous. 
Stopping times and tightness. 
\emph{Ann. Probability} \vol6:2 (1978),  335--340.
\MR{0474446}

\bibitem[Aldous(1993)]{AldousII} 
David Aldous.
The continuum random tree II. An overview.
\emph{Stochastic Analysis (Durham, 1990)}  \textbf{167} (1991),  23--70.
\MR{1166406}

\bibitem[Aldous(1993)]{AldousIII} 
David Aldous.
The continuum random tree III.
\emph{Ann. Probability}  \textbf{21}:1 (1993),  248--289.
\MR{1207226}

\bibitem[Aldous(2000)]{Aldous2000} 
David Aldous and Jim Pitman.
Inhomogeneous continuum random trees and the entrance boundary of the
additive coalescent.
\emph{Probab. Theory Related Fields}  \textbf{118}:4 (2000),  455--482.
\MR{1808372}

\bibitem{Bertoin}
Jean Bertoin.
Eternal additive coalescents and certain bridges with exchangeable increments. 
\emph{Ann. Probability} \vol{29} (2001),  344--360.
\MR{1825153}

\bibitem{BOH2023}
Gabriel Berzunza Ojeda and Cecilia Holmgren. 
Invariance principle for fragmentation processes derived from conditioned
stable Galton--Watson trees. 
\emph{Bernoulli} \vol{29}:4 (2023), 2745--2770.
\MR{4632119}

\bibitem{BHP2025}
Gabriel Berzunza Ojeda, Cecilia Holmgren and Paul Th\'evenin. 
Convergence of trees with a given degree sequence and of their associated laminations. 
\emph{arXiv:2111.07748} (2025+).

\bibitem[Billingsley(1999)]{Billingsley2ed}
Patrick Billingsley.
\book{Convergence of Probability Measures}.
2nd ed., 
\Wiley, New York, 1999.
\MR{0233396} 

\bibitem{Boucheron}
St\'ephane Boucheron, G\'abor Lugosi and Pascal Massart.
\emph{Concentration Inequalities}.
Oxford University Press, Oxford, 2013.
\MR{3185193} 

\bibitem{BroutinMarckert} 
Nicolas Broutin and Jean-Fran{\c c}ois Marckert.
A new encoding of coalescent processes: applications to the additive and
multiplicative cases.
\emph{Probab. Theory Related Fields} \vol{166}:1-2 (2016), 515--552.
\MR{3547745}

\bibitem{Camarri}
Michael Camarri and Jim Pitman.
Limit distributions and random trees derived from the birthday problem with unequal probabilities.
\emph{Electron. J. Probab.} \vol{5}:2 (2000), 18 pp.
\MR{1741774}

\bibitem{Delmas2018}
Jean-Fran\c cois Delmas, Jean-St\'ephane Dhersin and Marion Sciauveau.
Cost functionals for large (uniform and simply generated) random trees.
\emph{Electron. J. Probab.} \vol{23} (2018), Paper No.\ 87, 36 pp.
\MR{3858915}

\bibitem{Duquesne}
Thomas Duquesne.
A limit theorem for the contour process of conditioned Galton--Watson trees.
\emph{Ann. Probab.} \vol{31}:2 (2003), 996--1027.
\MR{1964956}

\bibitem{EthierKurtz}
Stewart N. Ethier and  Thomas G. Kurtz.
\emph{Markov Processes. Characterization and Convergence.} 
John Wiley \& Sons, Inc., New York, 1986. 
\MR{0838085}

\bibitem{FellerII}
William Feller.
\emph{An Introduction to Probability Theory and its Applications. Vol. II.}
John Wiley \& Sons, Inc., New York-London-Sydney, 1971. 
\MR{270403}

\bibitem{Gut}
Allan Gut.
\emph{Probability: A Graduate Course}.
2nd ed., Springer, New York, 2013. 
\MR{2977961} 


\bibitem{Svante2003} 
Svante Janson. 
The Wiener index of simply generated random trees. 
\emph{Random Structures Algorithms}, \vol{22} (2003), 337--358. 
\MR{1980963}

\bibitem[Janson, {\L}uczak and Ruci\'nski(2000)]{JLR}
Svante Janson, Tomasz \L uczak \& Andrzej Ruci\'nski.
\emph{Random Graphs}.
\Wiley, New York, 2000.

\bibitem{Kallenberg}
Olav Kallenberg.
\book{Foundations of Modern Probability.}
2nd ed., Springer, New York, 2002. 
\MR{1876169}

\bibitem{Kortchemski2013} 
Igor Kortchemski. 
A simple proof of Duquesne's theorem on contour processes of conditioned
Galton--Watson trees. 
\emph{S\'eminaire de Probabilit\'es XLV}, 
Lecture Notes in Math. \vol{2078}, pp. 537--558. 
Springer, Cham, 2013.
\MR{3185928}

\bibitem{Kortchemski} 
Igor Kortchemski.
Sub-exponential tail bounds for conditioned stable {B}ienaym\'e-{G}alton-{W}atson trees.
\emph{Probab. Theory Related Fields} \vol{168}:1-2 (2017), 1--40.
\MR{3651047}

\bibitem{Marzouk} 
Cyril Marzouk.
On scaling limits of random trees and maps with a prescribed degree sequence.
\emph{Ann. H. Lebesgue} \vol{5} (2022), 317--386.
\MR{4443293}

\bibitem{MeirMoon1970} 
Amram Meir and John W. Moon.
The distance between points in random trees.
\emph{J. Combinatorial Theory} \vol{8} (1970), 99--103.
\MR{263685}

\bibitem{Pitman}
Jim Pitman. 
\emph{Combinatorial Stochastic Processes}. 
{\'E}cole d'{\'E}t{\'e} de Probabilit{\'e}s de Saint-Flour
XXXII -- 2002.
Lecture Notes in Math. \vol{1875}, Springer-Verlag, Berlin, 2006. 
\MR{2245368}

\bibitem{Renyi1967} 
Alfr\'ed R\'enyi and George Szekeres.
On the height of trees.
\emph{J. Austral. Math. Soc.} \vol{7} (1967), 497--507.
\MR{219440}


\bibitem{Takacs1992} 
Lajos Tak\'acs. 
On the total heights of random rooted trees. 
\emph{J. Appl. Probab.} \vol{29}:3 (1992),  543--556.

\bibitem{Takacs1993} 
Lajos Tak\'acs. 
The asymptotic distribution of the total heights of random rooted trees. 
\emph{Acta Sci. Math. (Szeged)}, \vol{57}:1-4 (1993), 613--625. 
\MR{1243312}

\bibitem{Whitt}
Ward Whitt.
\emph{Stochastic-Process Limits}.
Springer-Verlag, New York, 2002.
\MR{1876437}

\end{thebibliography}
\end{document}